\documentclass {siamltex} 

\usepackage{amssymb}
\usepackage[T1]{fontenc}
\usepackage{amsfonts}
\usepackage[all]{xy}
\usepackage{times}

\newcommand {\apart} [1] {\kern 0.25em {} #1 {} \kern 0.25em}
\newcommand {\ball} [2] {B_{#1} (#2)}
\newcommand {\Ball} [2] {B_{#1} \left(#2\right)}
\newcommand {\bd} [1] {\mathop {\mbox {\rm bd}} \kern 0.1em (#1)}
\newcommand {\cl} [1] {\mathop {\mbox {\rm cl}} \kern 0.1em (#1)}
\renewcommand {\d} [1] {D_{#1}}
\newcommand {\differential} [1] {\, \mathop {\equiv}_{#1}^{\partial} \,}
\newcommand {\dist} [1] {\mathop {\mbox {\rm dist}} \kern 0.1em (#1)}
\newcommand {\Function} [1] {\smash {F^{(#1)}}}
\newcommand {\Inf} [1] {\inf_{\mbox {\substrut \scriptsize #1}}}
\newcommand {\Int} [1] {\mathop {\mbox {\rm int}} \kern 0.1em (#1)}
\renewcommand {\j} [1] {{\mathcal J}_{\kern -0.075em \it #1}}
\newcommand {\J} {{\mathbf J}}
\newcommand {\K} {{\mathbf K}}
\newcommand {\Lim} [1] {\lim_{\mbox {\substrut \scriptsize #1}}}

\newcommand {\Max} [1] {\max_{\mbox {\substrut \scriptsize #1}}}
\newcommand {\Min} [1] {\min_{\mbox {\substrut \scriptsize #1}}}

\newcommand {\mlb} {\smash {m^{(\ref {lem:ml})}}}
\newcommand {\note} [1] {\qquad \mbox {#1}}
\renewcommand {\r} {{\mathbf r}}
\newcommand {\rational} [1] {\, \mathop {\equiv}_{#1}^{\div} \,}
\newcommand {\Reals} {{\mathbb R}}
\newcommand {\reals} {{\mathbb R}}
\newcommand {\substrut} {{\normalsize \vrule depth0pt height1.5exwidth0pt}}
\newcommand {\Set} {{\mathcal S}}
\newcommand {\Sup} [1] {\sup_{\mbox {\substrut \scriptsize #1}}}

\newcommand {\TanX} {{\mathcal T}}
\newcommand {\TanY} {\Function 1}
\newcommand {\yone} {y_1}
\newcommand {\ytemp} {\hat y}

\newcommand {\colocated} {\smash {N^{(\ref {lem:colocated})}_{x_0} (\epsilon)}}
\newcommand {\Lipschitz} {\smash {N^{(\ref {lem:Lipschitz})}_{x_0}}}
\newcommand {\LipschitzForOther} [1] {\smash {N^{(\ref {cor:Lipschitz}.#1)}_{x_0}}}
\newcommand {\ml} {\smash {N^{(\ref {lem:ml})}_{x_0}}}
\newcommand {\MVTd} [1] {\smash {N^{(\ref {lem:mean})}_{y_0} (#1)}}
\newcommand {\MVTs} [1] {\smash {N^{(\ref {lem:mean})}_{x_0} (#1)}}

\newcommand {\MuF} {\mu_{F}^{}}
\newcommand {\Muf} [1] {\mu_{F^{(#1)}}^{}}
\newcommand {\Mu} [1] {\mu_{#1}^{}}

\newcommand {\prob} [1] {(P_{#1}^{})}
\newcommand {\probmax} [1] {(P_{#1})_{\, \max}}
\newcommand {\probmin} [1] {(P_{#1})_{\, \min}}

\newtheorem {dummy} {Dummy} [section]
\newtheorem {hypothesis} [dummy] {Hypothesis}

\begin{document}

\title{Asymptotic Equality of Minimal Perturbations Under Linearization}
\title{Asymptotic Equality of Minimal Perturbations to Roots\\ Under Linearization of the Equations}
\title{Asymptotic Equality of Minimal Perturbations to Roots}
\title{Asymptotic Equality Under Linearization of Minimal Perturbations to Roots for Parameterized Equations}
\title{Minimal Perturbations to Roots of Parameterized Equations are Asymptotically Equal Under Linearization}
\title{Minimal Perturbations to Roots of Parameterized Equations are Asymptotically Equal Under Linearization}
\title{Minimal Perturbations to Roots of Parameterized Equations are Asymptotically Equal Under Linearization}
\title{Minimal Perturbations to Roots of Parameterized Equations}

\author{Joseph F. Grcar\thanks{6059 Castlebrook Drive, Castro Valley, CA 94552 USA (jfgrcar@comcast.net).}}

\maketitle

\begin{abstract}
The size of minimal perturbations to roots of parameterized equations can be estimated reliably from linearizations of the equations.
\end{abstract}

\begin{keywords}
parametric optimization,
perturbation analysis
\end {keywords}

\begin {AMS}
90C31,
26B05,
41A29,
46N10
\end {AMS}


\setlength {\arraycolsep} {0.25em}
\setlength {\tabcolsep} {0.25em}
\renewcommand\bottomfraction{1.0}
\renewcommand\floatpagefraction{1.0}
\renewcommand\textfraction{0.0}
\renewcommand\topfraction{1.0}


\section {Introduction}
This paper offers a systematic way to answer the question: how much change must occur in a solution of equations to compensate for perturbations to the equations? Short of finding all the nearby roots of the new equations, the minimal change can be determined in an asymptotic sense by linearizing the equations and considering the dual problems. This conclusion is exhaustive because all nearby roots are considered, and strong because the asymptotics imply differential approximations. 

The asymptotic relationship is proved here. Companion papers make applications to differentiability of best approximations and to numerical analysis. 

\section {Approach}

\subsection {Introduction}
\label {sec:equivalence} 

Let the equations be $F(y, x) =0$ with the specific root $(y_0, x_0)$. The variable $x$ is regarded as the parameter so that $y$ depends on $x$ constrained by $F(y,x)=0$. The equations for $y$ may be underdetermined so $y$ may not be a function of $x$. Nevertheless, the size of minimal perturbations to $y_0$ is a function,
\begin {equation}
\label {eqn:mu-F}
\MuF (x) = \Min {$y : F(y, x) = 0$} \| y - y_0 \| \, .
\end {equation}
The idea is to study the value of this optimization problem by linearizing the equations. There are two requirements for the altered problems: 
\smallskip
\begin {enumerate}
\raggedright
\item The values of the simplified problems should mimic how $\MuF(x)$ varies with $x$.
\item Since $\MuF(x)$ is of interest when $x \approx x_0$, good mimicry is needed near $x_0$. 
\end {enumerate}
\smallskip
The novelty of the present approach is to formalize these requirements by equivalence relations, $\equiv$, among functions of $x$; two equivalences are chosen in section \ref {sec:relation}. Problem (\ref {eqn:mu-F}) is then altered by linearizing $F$; three linearizations, $\Function i$, are constructed in section \ref {sec:linearizations}. The bulk of the paper establishes equivalences $\MuF \equiv \Muf i$. For simplicity, the values of the altered problems are written $\Muf i = \Mu i$.

\subsection {Equivalence Relations}
\label {sec:relation}

The following equivalence relation is appropriate when differentiability at $x_0$ is the object of study.

\begin {definition} [Differential equivalence]
\label {def:differential}
The functions $f$ and $g$ defined on a neighborhood of $x_0 \in \Reals^n$ with values in $\reals^p$ are differentially equivalent at $x_0$ provided $f - g$ has a Fr\'echet derivative of $0$ at $x_0$, equivalently,
\begin {equation} 
\label {eqn:FrechetEquivalent}
{\textstyle f \differential {x_0} g}
\quad \Longleftrightarrow \quad
\Lim {$x \rightarrow x_0$} {\| f(x) - g(x) \| \over \| x - x_0 \|} \; =
\; 0 \, .
\end {equation}
\end {definition}

\begin {lemma}
\label {lem:FrechetEquivalent}
Differential equivalence is an equivalence relation. (This lemma is clear and not proved.)
\end {lemma}

If $g$ is an affine function, then equation (\ref {eqn:FrechetEquivalent}) becomes the definition for the Fr\'echet derivative of $f$ at $x_0$. In this way the differential properties of $f$ at $x_0$ are determined by the differential equivalence class. 

A simpler but stronger equivalence relation is that real-valued functions should be relatively closer as $x$ approaches $x_0$.

\begin {definition} [Asymptotic equality]
\label {def:rational}
The real-valued functions $f$ and $g$ defined on a neighborhood of $x_0 \in \Reals^n$ are asymptotically equal at $x_0$ provided for every $\epsilon > 0$ there is a neighborhood $N (\epsilon)$ of $x_0$ such that $x \in N(\epsilon)$ implies 
\begin {equation} 
\label {eqn:rational}
{\textstyle f \rational {x_0} g}
\quad \Longleftrightarrow \quad
(1 - \epsilon) g(x) \le f(x) \le (1 + \epsilon) g(x) \, .
\end {equation}
\vspace*{-3ex}
\end {definition}
\begin {lemma}
\label {lem:equivalent}
Asymptotic equality is an equivalence relation. (This lemma is clear and not proved.)
\end {lemma}

Asymptotic equality is stronger than differential equivalence. For example, all functions with vanishing derivatives at $0$ are differentially equivalent there, but two monomials $c_1 x^{n_1}$ and $c_2 x^{n_2}$ are asymptotically equal at $0$ if and only if they are equal. 

For the function $\MuF (x)$ in equation (\ref {eqn:mu-F}), asymptotic equality implies differential equivalence. The proof of this implication in lemma \ref {lem:Lipschitz} depends on a modified implicit function theorem in lemma \ref {lem:ift}, and on the Lipschitz continuity of $\MuF (x)$ at $x_0$. 

\begin {hypothesis}
\label {hyp:}
Hypothesis \ref {hyp:1}--\ref {hyp:4} are used throughout this paper, while \ref {hyp:5} or \ref {hyp:6} are used occasionally. 

\begin {enumerate}
\item \label {hyp:1} Norms are given for $\Reals^m$, $\Reals^n$ and $\Reals^p$.
\item ${\mathcal D} \subseteq \Reals^m \times \Reals^n$ is a neighborhood of $(y_0, x_0)$.
\item \label {hyp:3} $F : {\mathcal D} \rightarrow \Reals^p$ is continuously Fr\'echet differentiable.
\item \label {hyp:4} $F (y_0, x_0) = 0$.
\item \label {hyp:5} $\d1 F (y_0, x_0) : \Reals^m \rightarrow \Reals^p$ is onto.
\item \label {hyp:6} $\d2 F (y_0, x_0) : \Reals^n \rightarrow \Reals^p$ is one-to-one.
\end {enumerate}
\end {hypothesis}

\newcommand {\hypotheses} [1] {\ref {hyp:} (\ref{hyp:1}--\ref {hyp:#1})}

\begin {lemma} 
[Modified implicit function theorem]
\label {lem:ift} 
Under hypotheses \hypotheses 5, there is a neighborhood $N$ of $x_0$ and a Fr\'echet differentiable function $\phi : N \rightarrow \Reals^m$ with $\phi (x_0) = y_0$ and $F(\phi(x), x) = 0$ for all $x \in N$.
\end {lemma}

\newcommand {\spann} [1] {(\Reals^m)^{(#1)}}

\begin {proof}
The proof applies the usual theorem, which requires that $\d1 F (y_0, x_0)$ be one-to-one. In the present case the mapping is onto, so there are $p$ vectors in $\Reals^m$ that map to linearly independent vectors in $\Reals^p$, and there are $m-p$ additional vectors that complete a basis for $\Reals^m$. Let $y = y^{(p)} + y^{(m-p)}$ be the decomposition of $y \in \Reals^m$ into the subspaces spanned by the respective sets of basis vectors. $\d1 F (y_0, x_0)$ restricted to $\spann p$ is one-to-one. This fact and hypotheses \hypotheses 4 suffice to invoke the implicit function theorem for the function defined by $\smash {F(y + y_0^{(m-p)}, x)}$ on the domain ${\mathcal D} \cap [\spann p \times \Reals^n]$. There is a neighborhood $N$ of $x_0$ in $\Reals^n$ on which there is a continuously differentiable function $\phi: N \rightarrow (\Reals^m)^{(p)}$ such that $\phi (x_0) = \smash {y_0^{(p)}}$ and $\smash {F (\phi(x) + y_0^{(m-p)}, x) = 0}$. The implicit function in the lemma is given by $\smash {\phi(x) + y_0^{(m-p)}}$. 
\end {proof}

\begin {lemma} 
[Existence of $\MuF (x)$ and properties]
\label {lem:Lipschitz} 
Under hypotheses \hypotheses 5, there is a constant $L > 0$ and a neighborhood $\Lipschitz$ of $x_0$ where the function $\MuF (x)$ of equation (\ref {eqn:mu-F}) exists, and $\MuF (x) \le L \| x - x_0 \|$. Further, for any function $f$, 
\begin {displaymath}
\textstyle
f \rational {x_0} \MuF \quad \Rightarrow \quad f \differential {x_0} \MuF \, .
\end {displaymath}
\end {lemma}
\vspace {-3ex}
\begin {proof}
Hypotheses \hypotheses 5 suffice to invoke the version of the implicit function theorem in lemma \ref {lem:ift}: $x_0$ has a neighborhood $N$ on which there is a continuously differentiable function $\phi: N \rightarrow \Reals^m$ such that $(\phi (x), x)$ is always a root of $F$. Thus the minimization problems for $\MuF (x)$ have feasible points for all $x \in N$. The feasible sets are closed because $F$ is continuous, so the minimal distance to $y_0$ is attained because the spaces have finite dimension. This means $\MuF$ is well defined on $N$. Since $\phi$ is continuously differentiable, it is Lipschitz continuous on compact sets. Choose a compact neighborhood $\Lipschitz \subseteq N$ with Lipschitz constant $L$. Thus $\MuF (x) \le \| \phi(x) - y_0 \| = \| \phi(x) - \phi(x_0) \| \le L \| x - x_0 \|$ for every $x$ in the neighborhood. 

Given $\epsilon > 0$, let $N(\epsilon)$ be the neighborhood in definition \ref {def:rational} for $f \rational {x_0} \MuF$. If $x \in N(\epsilon) \cap \smash{\Lipschitz}$, then $(1 - \epsilon) \MuF (x) \le f(x) \le (1 + \epsilon) \MuF (x)$ by the equivalence, so $| f(x) - \MuF (x) | \le \epsilon \, \MuF (x) \le \epsilon L \| x - x_0 \|$ and thus the limit in equation (\ref {eqn:FrechetEquivalent}) vanishes.
\end {proof}

\subsection {Linearized Problems with Equivalent Minimal Perturbations}
\label {sec:linearizations}

It is instructive to compare the present situation with the implicit function theorem. Under hypotheses \hypotheses 4 and if $\d1 F (y_0, x_0): \Reals^m \rightarrow \Reals^p$ is invertible, then some roots of $F(y,x) = 0$ are given by a smooth parameterization $(\phi(x), x)$. These roots can be located to first order in $x-x_0$ by considering the linearization,
\begin {equation}
\label {eqn:iftcase}
F(\phi(x),x) = 0 \quad \Rightarrow \quad \left[ \d1 F (y_0, x_0) \, D \phi (x_0) + \d2 F(y_0, x_0) \right] (x - x_0) = 0 \, .
\end {equation}
The parameterized roots are approximated by,
\begin {displaymath}
\phi(x) - y_0 \approx - \left[ \d1 F (y_0, x_0) \right]^{-1} \left[ \d2 F(y_0, x_0) \right] (x - x_0) \, .
\end {displaymath}
In contrast, if $\d1 F (y_0, x_0): \Reals^m \rightarrow \Reals^p$ is not invertible, the smallest change $y - y_0$ as a function of $x$ can still be approximated from the linearizations $\Function i$ of $F$ in \textbf {Table \ref {tab:simplifications}}. 

\begin {table} [h!]
\caption {Linearizations of the function $F$ at $(y_0, x_0)$. The notation is $\Delta y = y - y_0$ and $\Delta x = x - x_0$.}
\label {tab:simplifications}
\begin {center}
\small
\renewcommand {\arraystretch} {1.33}
\begin {tabular} {c c c || c r c l c |}
\hline
&0&&& $F(y, x)$&&&\\
\hline
&1&&& $\displaystyle \Function1 (y, x)$& $=$& $\displaystyle \d1 F (y_0, x) \, \Delta y + F (y_0, x)$&\\
&2&&& $\displaystyle \Function2 (y, x)$& $=$& $\displaystyle \d1 F (y_0, x_0) \, \Delta y + F (y_0, x)$&\\
&3&&& $\displaystyle \Function3 (y, x)$& $=$& $\displaystyle \d1 F (y_0, x_0) \, \Delta y + \d2 F(y_0, x_0) \, \Delta x$&\\
\hline
\end {tabular}
\end {center}
\end {table}

The different linearizations have different uses. For example, $\Function 1$ and (\ref {eqn:mu-F}) do not require $x_0$. The several approximations are treated in a progression of equivalences for $F$ and $\Function 1$, then $\Function 1$ and $\Function 2$, and so on. The last $\Function 3$ is the full linearization (\ref {eqn:iftcase}) of the implicit function theorem. The proof of asymptotic equality for $\Function i$ and $\Function {i+1}$ is carried out with the dual mathematical programs. All the optimization problems are listed in \textbf {Table \ref {tab:equivalences}}, and the network of equivalences to be established is shown in \textbf {Figure \ref {fig:proofs}}.

If $F$ satisfies hypotheses \hypotheses {5}, then all the functions of Table \ref {tab:simplifications} satisfy the same hypotheses, so they also satisfy the conclusions of lemma \ref {lem:Lipschitz}. 

\begin {corollary} 
[Existence of $\Mu i (x)$ and properties]
\label {cor:Lipschitz}
Under hypotheses \hypotheses 5 for $F$, for each function $\Function i$ of Table \ref {tab:simplifications} there is a constant $L_i > 0$ and a neighborhood $\LipschitzForOther i$ of $x_0$ where the following distance function is well defined
\begin {equation}
\label {eqn:mu-for-other}
\Mu i (x) = \Min {$y : \Function i (y, x) = 0$} \| y - y_0 \| \, ,
\end {equation}
and $\Mu i (x) \le L_i \| x - x_0 \|$. Further, for any function $f$, 
\begin {displaymath}
\textstyle
f \rational {x_0} \Mu i \quad \Rightarrow \quad f \differential {x_0} \Mu i \, .
\end {displaymath}
\end {corollary}
\vspace {-3ex}
\begin {proof} The linearizations satisfy the same hypotheses as $F$, so lemma \ref {lem:Lipschitz} applies. 
\end {proof}

\newcommand {\centerbox} [1] {\begin {minipage} {6em} \begin {center}
\vrule depth.66ex height1.75ex width0pt \scriptsize #1 \end {center}
\end {minipage}}
\newcommand {\leftbox} [1] {\mbox {\vrule depth.66ex height1.75ex width0pt \scriptsize \kern 0.5em #1}}
\newcommand {\leftparbox} [1] {\kern 0.4em\parbox {11em} {\vrule depth.66ex height1.75ex width0pt \scriptsize #1}}
\newcommand {\rightbox} [1] {\mbox {\vrule depth.66ex height1.75ex width0pt \scriptsize #1 \kern 0.0em}}

\begin {table} 
\caption {Optimization problems parameterized by $x$ and their duals. The values of problems $\prob{}$, $\prob1$, $\prob2$ are asymptotically equal at $x_0$ under hypotheses \hypotheses {5}. The value of $\prob3$ is differentially equivalent to the others under these hypotheses, and is asymptotically equal under hypotheses \hypotheses {6}. See Table \ref {tab:matrix} for the problems in matrix notation. In these formulas, $\Delta x = x - x_0$ and $\Delta y = y - y_0$.}
\label {tab:equivalences}
\small
\renewcommand {\substrut} {{\normalsize \vrule depth0pt height2ex
width0pt}}
\renewcommand {\Max} [1] {\max_{\makebox[5em][l]{\substrut \scriptsize #1}}}
\renewcommand {\Min} [1] {\min_{\makebox[5em][l]{\substrut \scriptsize #1}}}
\newcommand {\minobject} {\| \Delta y \|}
\newcommand {\tablestrut} {\vrule depth4.25ex height3.0ex width0pt}
\arraycolsep = 0.75em
\vspace {-2ex}
\begin {displaymath}
\begin {array} {| c | c | c | l l |}
\multicolumn {2} {c} {\relax}& \multicolumn {1} {c} {\hspace{-1em}\mbox {constraint}\hspace{-1em}}\\
\mbox {name}& \mbox {value}& \mbox {function}& \multicolumn{2}{c |}{\mbox {minimization form} \hfill \mbox {dual, maximization form}}\\
\hline \hline
\tablestrut \prob{}&
\MuF (x)& F(y,x)&
\displaystyle 
\Min {$y : F (y, x) = 0$} \minobject&
\\ \hline
\tablestrut \prob1&
\Mu1 (x)& \Function 1 (y, x)&
\multicolumn {2} {c |}
{
\displaystyle 
\Min {$y : \d1 F (y_0, x) \Delta y + F (y_0, x) = 0$} \minobject
\hspace*{7em}
\Max {$f : \| \d1 F (y_0, x)^* f \| \le 1$} f (F (y_0, x))
}
\\ \hline
\tablestrut \prob2&
\Mu2 (x)& \Function 2 (y, x)&
\multicolumn {2} {c |}
{
\displaystyle 
\Min {$y : \d1 F (y_0, x_0) \Delta y + F (y_0, x) = 0$} \minobject
\hfill
\Max {$f : \| \d1 F (y_0, x_0)^* f \| \le 1$} f (F (y_0, x))
}
\\ \hline
\tablestrut \prob3&
\Mu3 (x)& \Function 3 (y, x)&
\multicolumn {2} {c |}
{
\displaystyle
\Min {$y : D F (y_0, x_0) (\Delta y, \Delta x) = 0$} \minobject
\hfill
\Max {$f : \| \d1 F (y_0, x_0)^* f \| \le 1$} f (\d2 F (y_0, x_0) \Delta x)
}
\\ \hline
\end {array}
\end {displaymath}
\end {table}

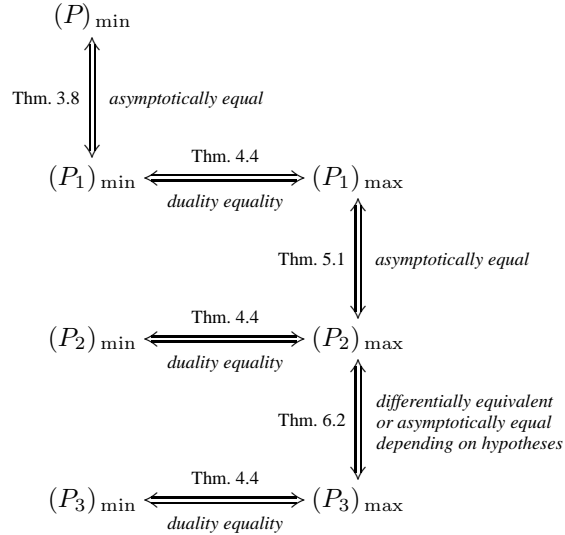
\begin {figure} 
\begin {displaymath}
\xymatrix@C=6em@R=10ex {
\probmin{}\ar@{<=>}[d]_{\rightbox{Thm.\ \ref {thm:first}}}^{\leftbox{\it asymptotically equal}}\\
\probmin1\ar@{<=>}[r]^{\centerbox{Thm.\ \ref {thm:duality}}}_{\centerbox{\it duality equality}}&
\probmax1\ar@{<=>}[d]_{\rightbox{Thm.\ \ref {thm:second}}}^{\leftbox{\it asymptotically equal}}\\
\probmin2\ar@{<=>}[r]^{\centerbox{Thm.\ \ref {thm:duality}}}_{\centerbox{\it duality equality}}&
\probmax2\ar@{<=>}[d]_{\rightbox{Thm.\ \ref {thm:third}}}^{\leftparbox{\it differentially equivalent\\ or asymptotically equal\\ depending on hypotheses}}\\
\probmin3\ar@{<=>}[r]^{\centerbox{Thm.\ \ref {thm:duality}}}_{\centerbox{\it duality equality}}& \probmax3
}
\end {displaymath}
\caption {Where and how the equivalences of Table \ref {tab:equivalences} are proved.}
\label {fig:proofs}
\end {figure}

\section {First Equivalence, $\probmin{} \equiv \probmin1$}
\label {sec:first}

The preparations to establish the first equivalence are the most elaborate in this paper. Several aspects of the difference between $F$ and the tangent function for $y$, $\TanY$, are uniform in $x - x_0$: mean values, Fr\'echet differentials, and level sets. The first equivalence thus requires giving a uniform parameterization to many basic concepts in real analysis, which are indicated in \textbf {Figure \ref {fig:first}}. The mean value theorem and Fr\'echet quotient are discussed in section \ref {sec:mean}, the matrix lower bound is in section \ref {sec:mlb}, level sets are in section \ref {sec:level}, and finally the proof of the first equivalence is in section \ref {sub:first}.

\begin {figure}
\newcommand {\mybox} [2] {\begin {minipage} {#1} \begin {center} \vrule depth0pt height1.75ex width0pt \scriptsize #2 \end {center} \end {minipage}}
\begin {displaymath}
\xymatrix {
\mybox {8em} {\textit {uniformly} parameterized\break mean value theorem,\break Lemma \ref {lem:mean}}\ar[r]\ar[rd]
& \mybox {8em} {\textit {uniformly} approximating Fr\'echet differential,\break Corollary \ref {cor:Frechet}}\ar[d]& 
\mybox {7em} {matrix lower bound,\break Definition \ref {def:lower}}\ar[d]
\\
\mybox {6em} {Lipschitz\break continuity of $\MuF$, Lemma \ref {lem:Lipschitz}}\ar[dr]\ar[d]& 
\mybox {7em} {\textit {uniformly} colocated level sets,\break Lemma \ref {lem:colocated}}\ar[d]& 
\mybox {8em} {\textit {uniformly} bounded below partial derivatives,\break Lemma \ref {lem:ml}}\ar[l]
\\
\mybox {6em} {Lipschitz\break continuity of $\Mu 1$, Corollary \ref {cor:Lipschitz}}\ar[r]
& \mybox {6em} {1$^{\rm st}$ equivalence,\break Theorem \ref {thm:first}}
}
\end {displaymath}
\caption {Dependencies for the proof of the first equivalence.}
\label {fig:first}
\end {figure}
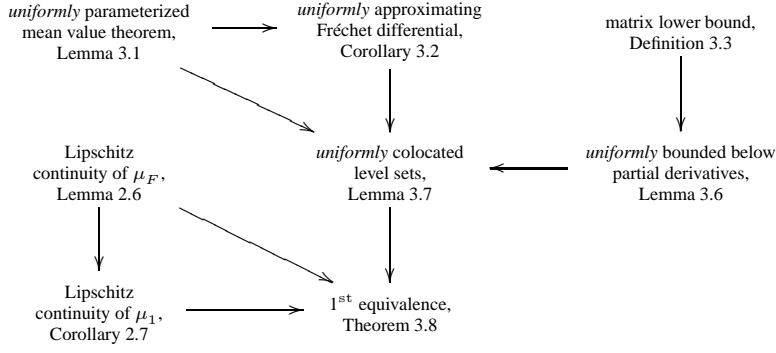

\subsection {Uniformly Parameterized Mean Value Theorem}
\label {sec:mean}

It is well known that if $f$ is continuously differentiable, then for every $y_3$ and every $\epsilon > 0$ there is a neighborhood $N_{y_3} (\epsilon)$ of $y_3$ where
\begin {equation} \label {eqn:keyLemma}
y_1, y_2 \in N_{y_3} (\epsilon) \kern 0.5em \Rightarrow \kern0.5em \| f
(y_1) - f (y_2) - D f (y_3) (y_1 - y_2) \| \le \epsilon \, \| y_1 - y_2
\| \, .
\end {equation}
This serves as a mean value theorem in multiple dimensions. Luenberger \cite [p.\ 212] {Luenberger1969} remarks that it has been discussed many times. Bartle \cite [p.\ 377] {Bartle1976} calls (\ref {eqn:keyLemma}) the ``key lemma'' for theorems like the implicit function theorem. Ortega and Rheinboldt \cite [p.\ 72] {Ortega1970} show that (\ref {eqn:keyLemma}) is equivalent to the continuity of the derivative. Here, this surrogate mean value theorem is generalized to parameterized functions. 

\begin {lemma} 
[Uniformly parameterized mean value theorem]
\label {lem:mean}
Under hypotheses \hypotheses 5, for every $\epsilon > 0$ there is a neighborhood $\MVTd \epsilon \, \times \, \MVTs \epsilon \subseteq {\mathcal D}$ such that for all $y_1, y_2, y_3 \in \MVTd \epsilon$ and $x \in \MVTs \epsilon$,
\begin {equation} 
\label {eqn:mean}
\| F (y_1, x) - F (y_2, x) - \d1 F (y_3, x) (y_1 - y_2) \| \le \epsilon \, \| y_1 - y_2 \| \, .
\end {equation}
\end {lemma}
\vspace*{-3ex}
\begin {proof} The topology of the product space $\Reals^m \times \Reals^n$ can be generated from the products of the open sets, so it is possible to choose a compact, convex neighborhood $Y$ around $y_0$, and a compact neighborhood $X$ around $x_0$, so that $Y \times X \subseteq {\mathcal D}$. All norms for a finite dimensional space generate the same topology, so without loss of generality let the norm for $\Reals^m \times \Reals^m \times \reals^n$ be $\max \{ \| y_1 \|, \| y_2 \|, \| x \| \}$. Since $\d1 F (y, x)$ is continuous, hence $g(y_1, y_2, x) = \d1 F (y_1, x) - \d1 F (y_2, x)$ is uniformly continuous on the compact set $K = Y \times Y \times X$. The uniform continuity means, for every $\epsilon > 0$ there is a $\delta (\epsilon) > 0$ so that if $(y_1, y_2, x), (y_1^\prime, y_2^\prime, x^\prime) \in K$ with $\max \, \{ \| y_1 - y_1^\prime \|, \| y_2 - y_2^\prime \|, \| x - x^\prime \| \} \le \delta (\epsilon)$, then $\| g(y_1, y_2, x) - g(y_1^\prime, y_2^\prime, x^\prime) \| \le \epsilon$. 

Choose the neighborhoods in the statement of the lemma to be $\MVTd \epsilon = \ball {y_0} {\delta (\epsilon)} \cap Y$ and $\MVTs \epsilon = \ball {x_0} {\delta (\epsilon)} \cap X$. Note, these sets are convex. If $y_1$, $y_2$, $y_3$ and $x$ are from the respective sets, then 
\begin {eqnarray*}
\| \d1 F (t y_1 + (1-t) y_2, x) - \d1 F (y_3, x) \|& =& \| g (t y_1 + (1-t) y_2, y_3, x) \|\\
& =& \| g (t y_1 + (1-t) y_2, y_3, x) - g(y_0, y_0, x_0) \|\\
& \le& \epsilon \, .
\end {eqnarray*}
It is well known from \cite [p.\ 376, lemma 41.3] {Bartle1976} or from \cite [p.\ 70, lemma 3.2.5] {Ortega1970} that if $D \subseteq \Reals^m$ is a convex, open set, and if $f : D \rightarrow \Reals^p$ is continuously differentiable, then for any $y_1, y_2, y_3 \in D$,
\begin {displaymath}
\| f(y_1) - f(y_2) - Df (y_3) (y_1 - y_2) \| \le \Sup {$0 \le t \le 1$} \| Df (t y_1 + (1-t) y_2) - Df (y_3) \| \, \| y_1 - y_2 \| \, .
\end {displaymath}
Applying this inequality to the parameterized function $F (y, x)$ for the previously chosen $y_1$, $y_2$, $y_3$, $x$ gives
\begin {eqnarray*}
&& \hspace*{-2em} \| F (y_1, x) - F (y_2, x) - \d1 F (y_3, x) (y_1 - y_2) \|\\
\noalign {\smallskip}
& \le& \Sup {\hspace{-1em}$0 \le t \le 1$\hspace{-1em}} \quad \| \d1 F (t y_1 + (1-t) y_2, x) - \d1 F (y_3, x) \| \, \| y_1 - y_2 \|\\
\noalign {\smallskip}
& \le& \epsilon \, \| y_1 - y_2 \| \, .
\end {eqnarray*}
\end {proof}

Lemma \ref {lem:mean} gives conditions under which the Fr\'echet differential for $y$ is uniformly approximating with respect to the parameter $x$.

\begin {corollary} [Uniformly approximating differential]
\label {cor:Frechet}
The neighborhoods of lemma \ref {lem:mean} also satisfy, for all $y \in \MVTd \epsilon$ and $x \in \MVTs \epsilon$,
\begin {equation} 
\label {eqn:Frechet}
\| F (y, x) - \TanY (y, x) \| \le \epsilon \, \| y - y_0 \| \, ,
\end {equation}
where $\TanY (y, x)$ is the parameterized tangent function of Table \ref {tab:simplifications}.
\end {corollary}

\begin {proof} 
Choose $y_1 = y$, $y_2 = y_0$ and $y_3 = y_0$ so that the formula in equation (\ref {eqn:mean}) becomes
\begin {eqnarray*}
F (y_1, x) - F (y_2, x) - \d1 F (y_3, x) (y_1 - y_2) & =& F (y, x) - F (y_0, x) - \d1 F (y_0, x) (y - y_0)\\
\noalign {\smallskip}
& =& F (y, x) - \TanY (y, x) \, .
\end {eqnarray*}
\end {proof}

\subsection {Matrix Lower Bound}
\label {sec:mlb}

The matrix lower bound, $\| A \|_\ell$, is analogous to the matrix norm but with reversed inqualities. The following are from \cite [p.\ 205, def.\ 2.1 and lem.\ 2.2; p.\ 212, cor.\ 4.3] {Grcar2010}.

\begin {definition} [Matrix lower bound]
\label {def:lower}
Let $A$ be a nonzero matrix. The matrix lower bound, $\| A \|_\ell$, is the largest of the numbers, $m$, such that for every $y$ in the column space of $A$, there is some $x$ with $A x = y$ and $m \, \| x \| \le \|y\|$. 
\end {definition}

\begin {lemma} 
\label {lem:exists}
The matrix lower bound exists and is positive for every nonzero matrix.
\end {lemma}

\begin {lemma} 
\label {lem:continuous}
The matrix lower bound is continuous on the open set of full rank matrices.
\end {lemma}

The present use of the lower bound is in the following lemma.

\begin {lemma} 
[Uniform lower bounds for partial derivatives]
\label {lem:ml}
Under hypotheses \hypotheses 5, there is a neighborhood $\ml$ of $x_0$ where $\d1 F (y_0, x) : \Reals^m \rightarrow \Reals^p$ is onto for every $x \in \ml$. There is also a number $\mlb > 0$ such that every $x \in \ml$ and $u \in \Reals^p$ have some $w \in \Reals^m$ (which depends on $x$ and $u$) so that $\d1 F (y_0, x) w = u$ and $\mlb \| w \| \le \| u \|$.
\end {lemma}

\begin {proof}
Choose some bases for $\Reals^m$ and $\Reals^p$ so that these spaces are represented by real column vectors. The linear transformations $\d1 F (y_0, x)$ are then represented by $p \times m$ matrices, $A(x)$. By Hypothesis \ref {hyp:} (\ref {hyp:3}) $F$ is continuously differentiable and (\ref {hyp:5}) $\d1 F (y_0, x_0)$ is onto, which mean $A(x)$ is a continuous function of $x$ and the column space of $A(x_0)$ is all of $\Reals^p$, or equivalently $A(x_0)$ has full row rank. For a matrix $M$ to have full row rank means $\det (M M^t)$ does not vanish. The determinant is a continuous function of the matrix, so $A(x_0)$ has a neighborhood of matrices $N_{A(x_0)}$ all of which have full row rank. From the continuity of $A(x)$, there is a neighborhood $N_{x_0}$ for which all matrices lie in $N_{A(x_0)}$. Hence for all $x \in N_{x_0}$ the mappings $\d1 F (y_0, x)$ are onto, or equivalently the column space of each matrix $A(x)$ is all of $\Reals^p$. 

Choose a compact neighborhood $\ml \subseteq N_{x_0}$. Since $\| A(x) \|_\ell$ is continuous and positive on $N_{x_0}$ by lemmas \ref {lem:exists} and \ref {lem:continuous}, $\| A(x) \|_\ell$ is uniformly bounded below on $\ml$ by some $\mlb > 0$. If $x \in \ml$ and $u \in \Reals^p$, then since the column space of $A(x)$ is all of $\Reals^p$, by definition \ref {def:lower} there is $w \in \Reals^m$ so $\d1 F (y_0, x) w = A(x) w = u$ and $\| A(x) \|_\ell \, \| w \| \le \| u \|$. Further, $\mlb \le \| A(x) \|_\ell$ by the choice of $\ml$.
\end {proof}

\subsection {Uniformly Colocated Level Sets}
\label {sec:level}

Suppose $D$ is an open set in $\Reals^m$, on which $f : D \rightarrow \Reals^p$ is continuously differentiable. By analogy with real-valued functions, the set $f^{-1} (a)$ may be called a level set of $f$. 
It is possible to make a geometric comparison between the level sets of $f$ and those of its tangent function at $y_0$. For functions such as $F$ that vary smoothly with a parameter, the distance between the corresponding level sets is uniformly bounded with respect to changes in the parameter. The proof is a modification of a construction apparently due to L.\ M.\ Graves \cite {Graves1950}, see also \cite [p.\ 378, theorem 41.6] {Bartle1976}. 

\begin {lemma} 
[Uniformly colocated level sets]
\label {lem:colocated}
Under hypotheses \hypotheses 5, for every $\epsilon > 0$ there is a radius $r(\epsilon) > 0$ and a neighborhood $\colocated$ of $x_0$ so $\cl {\ball {y_0} {r(\epsilon)}} \times \colocated \subseteq {\mathcal D}$. For each pair $(y, x) \in \ball {y_0} {r(\epsilon) / (1 + \epsilon)} \times \colocated$:
\medskip
\begin {center}
\renewcommand {\tabcolsep} {0.5em}
\begin {tabular} {c | c | c | c |}
\vrule depth1ex height0ex width0pt & {\rm (a)} there exists& {\rm (b)} with& {\rm (c)} and with\\
\hline 
\vrule depth1ex height2.5ex width0pt {\rm (1)}& $\yone \in \ball {y_0} {r(\epsilon)}$& $\TanY (\yone, x) = F (y, x)$& $\| \yone - y \| \le \epsilon \, \| y - y_0 \|$\\ 
\vrule depth1.25ex height2.5ex width0pt {\rm (2)}& $y_F^{} \in \cl {\ball {y_0} {r(\epsilon)}}$& $F (y_F^{}, x) = \TanY (y, x)$& $\| y_F^{} - y \| \le \epsilon \, \| y - y_0 \|$\\
\hline
\end {tabular}
\end {center}
\medskip
\end {lemma}

\begin {proof}
Lemma \ref {lem:colocated} has the first of the two most complicated proofs in this paper. Let $\delta = \epsilon / (1 + \epsilon) < 1$. Let $\mlb$ be the lower bound for the neighborhood $\ml$ in lemma \ref {lem:ml}. Choose a radius $r (\epsilon) > 0$ so that 
\begin {equation}
\label {eqn:r(e)}
\cl {\ball {y_0} {r (\epsilon)}} \subseteq \MVTd {\delta \, \mlb} \, .
\end {equation}
The neighborhoods from which the lemma is allowed to choose $y$ and $x$ are
\begin {eqnarray}
\label {eqn:choiceofy}
y \in \ball {y_0} {r (\epsilon) / (1+ \epsilon)} \subseteq \cl {\ball {y_0} {r(\epsilon)}}& \subseteq& \MVTd {\delta \, \mlb} \, ,\\
\noalign {\medskip}
\label {eqn:choiceofx}
x \in \colocated := \ml \cap \MVTs {\delta \, \mlb}& \subseteq& \MVTs {\delta \, \mlb} \, .
\end {eqnarray}
Note the product $\ball {y_0} {r(\epsilon)} \times \colocated$ is a subset of $\MVTd {\delta \, \mlb} \times \MVTs {\delta \, \mlb}$ in $\mathcal D$ by lemma \ref {lem:mean}.

(Part 1.) Because $\d1 F (y_0, x) : \Reals^m \rightarrow \Reals^p$ is onto, the range of the transformation contains the vector $F (y, x) - \TanY (y, x)$, and because $x \in \ml$ by (\ref {eqn:choiceofx}), lemma \ref {lem:ml} finds a $\ytemp$ with
\begin {eqnarray} 
\label {eqn:equality}
\d1 F (y_0, x) \, \ytemp& =& F (y, x) - \TanY (y, x) \, ,\\
\noalign {\medskip} 
\label {eqn:inequality}
\mbox {and} \quad \mlb \, \| \ytemp \|& \le& \| F (y, x) - \TanY (y, x) \| \, .
\end {eqnarray}
Let $y_1 = \ytemp+ y$ so $\ytemp = \yone - y$. The equality (\ref {eqn:equality}) and some algebra imply 
\begin {eqnarray*}
\TanY (\yone, x)
& =& \d1 F (y_0, x) (\yone - y_0) + F(y_0, x) 
\note {by definition of $\TanY$ in Table \ref {tab:simplifications}}\\
& =& \big[ \d1 F (y_0, x) (\yone - y) \big] + \big[\d1 F (y_0, x) (y - y_0) + F(y_0, x) \big]
\note {inserting $\pm y$}\\
& =& \big[ F (y, x) - \TanY (y, x) \big] + \TanY (y, x)
\note {by (\ref {eqn:equality}) and by definition of $\TanY$}\\
& =& F (y, x) \,
\end {eqnarray*}
which is part (1b). Further,
\begin {eqnarray*} 
\| \yone - y \| 
& \le& \displaystyle {\| F (y, x) - \TanY (y, x) \| \over \mlb}
\note {from (\ref {eqn:inequality})}\\ \noalign {\medskip}
& \le& \displaystyle {(\delta \, \mlb) \, \| y - y_0 \| \over \mlb}
\note {by (\ref {eqn:Frechet}) and $y \in \MVTd {\delta \, \mlb}$ in (\ref {eqn:choiceofy})}\\ \noalign {\smallskip}
& =& \delta \, \| y - y_0 \|\\ \noalign {\smallskip}
& <& \epsilon \, \| y - y_0 \| 
\note {by the choice $\delta = \epsilon / (1 + \epsilon)$},
\end {eqnarray*}
which is part (1c). Finally,
\begin {eqnarray*}
\| \yone - y_0 \|
& \le& \| \yone - y \| + \| y - y_0 \|\\
& \le& \epsilon \, \| y - y_0 \| + \| y - y_0 \|
\note {from (1c)}\\
& =& (1 + \epsilon) \, \| y - y_0 \|\\
& <& r (\epsilon)
\note {from the choice $y \in \ball {y_0} {r (\epsilon) / (1+ \epsilon)}$ in (\ref {eqn:choiceofy})}.
\end {eqnarray*}
Therefore $\yone \in \ball {y_0} {r (\epsilon)}$, which is part (1a).

(Part 2.) Let $y_1 = y$ from (\ref {eqn:choiceofy}). This $y_1$ and $y_0$ begin a sequence $\{ y_n \}$ to be built subject to the conditions:
\begin {displaymath}
\begin {array} {r l l}
(1_n)& \displaystyle \| y_{n+1} - y_n \| \le \delta^n \, \| y - y_0 \| \, ,\\
\noalign {\medskip}
(2_n)& \displaystyle \| F (y_{n+1}, x) - \TanY (y, x) \| \le (\delta \, \mlb) \, \| y_{n+1} - y_n \| \, .
\end {array}
\end {displaymath} 
Condition $(1_0)$ is just $\| y - y_0 \| \le \| y - y_0 \|$. Condition $(2_0)$ is (\ref {eqn:Frechet}) in corollary \ref {cor:Frechet} which is applicable by the choices of $y_1 = y$ and $x$ in equations (\ref {eqn:choiceofy}) and (\ref {eqn:choiceofx}).

Suppose $y_0$, $y_1$, $\ldots,$ $y_k$ have been constructed to satisfy $(1_n)$ and $(2_n)$ for $0 \le n \le k - 1$. The selection of $y_{k+1}$ proceeds as for $y_1$ in the first half of the proof. Again because $\d1 F (y_0, x) : \Reals^m \rightarrow \Reals^p$ is onto, the transformation maps to $- \left[ F (y_k, x) - \TanY (y, x) \right]$, and because $x \in \ml$ by (\ref {eqn:choiceofx}), it is possible to invoke lemma \ref {lem:ml} to find a $\ytemp$ with
\begin {eqnarray} 
\label {eqn:equality2}
\d1 F (y_0, x) \, \ytemp& =& - \left[ F (y_k, x) - \TanY (y, x) \right] \, ,\\
\noalign {\medskip} 
\label {eqn:inequality2}
\mbox {and} \quad \mlb \, \| \ytemp \|& \le& \| F (y_k, x) - \TanY (y, x) \| \, .
\end {eqnarray}
Let $y_{k+1} = \ytemp+ y_k$ so $\ytemp = y_{k+1} - y_k$. For this choice of $y_{k+1}$,
\begin {eqnarray*}
\| y_{k+1} - y_k \| 
& \le& \displaystyle {\| F (y_k, x) - \TanY (y, x) \| \over \mlb}
\note {from (\ref {eqn:inequality2})}\\ \noalign {\smallskip}
& \le& \displaystyle {(\delta \, \mlb) \, \| y_k - y_{k-1} \| \over \mlb}
\note {from $(2_{k-1})$}\\ \noalign {\smallskip}
& =& \delta \, \| y_k - y_{k-1} \|\\ \noalign {\smallskip}
& <& \delta^k \, \| y - y_0 \| 
\note {from $(1_k)$},
\end {eqnarray*}
which is $(1_k)$. Summing $(1_n)$ for $0 \le n \le k$ gives 
\begin {displaymath}
\| y_{k+1} - y_0 \| \le \sum_{n = 0}^k \| y_{n+1} - y_n \| \le {1 - \delta^{k+1} \over 1 - \delta} \, \| y - y_0 \| < (1 + \epsilon) \, \| y - y_0 \| \, ,
\end {displaymath}
which easily follows from the choice $\delta = \epsilon / (1 + \epsilon)$. This inequality combines with $y \in \ball {y_0} {r(\epsilon)/(1+\epsilon)}$ to place $y_{k+1} \in \ball {y_0} {r(\epsilon)} \subseteq \MVTd {\delta \, \mlb}$ from equation (\ref {eqn:r(e)}), and then $(y_{k+1}, x) \in {\mathcal D}$. Thus, the evaluation of $F (y_{k+1}, x)$ is well defined. Further,
\begin {eqnarray*}
&& \hspace*{-2em} \| F (y_{k+1}, x) - \TanY (y, x) \|\\ 
& =& \| F (y_{k+1}, x) - F (y_k, x) - \left\{ - \left[ F (y_k, x) - \TanY (y, x)\right]\right\} \|
\note {inserting $\pm F (y_k, x)$}\\
& =& \| F (y_{k+1}, x) - F (y_k, x) - \d1 F (y_0, x) (y_{k+1} - y_k) \|
\note {from (\ref {eqn:equality2})}\\
& \le& (\delta \, \mlb) \, \| y_{k+1} - y_k \| 
\note {from (\ref {eqn:mean})},
\end {eqnarray*}
which is $(2_{k})$.

In this way a sequence $\{ y_n \} \subseteq \ball {y_0} {r (\epsilon)}$ is constructed that satisfies conditions $(1_n)$ and $(2_n)$ for all $n$. The sequence is a Cauchy sequence by $(1_n)$, so it has a limit $y_F \in \cl {\ball {y_0} {r (\epsilon)}}$, which is part (2a). Passing to the limit in $(2_n)$ shows $F (y_F, x) = \TanY (y, x)$, which is part (2b). Summing $(1_n)$, now for $1 \le n \le k$, gives 
\begin {displaymath}
\| y_{k+1} - y \| = \| y_{k+1} - y_1 \| 
\le \sum_{n = 1}^k \| y_{n+1} - y_n \| \le \delta \, {1 - \delta^k \over 1 - \delta} \, \| y - y_0 \| \, ,
\end {displaymath}
which in the limit becomes (2c), $\| y_F - y \| \le \delta (1 - \delta)^{-1} \| y - y_0 \| = \epsilon \| y - y_0 \|$.
\end {proof}

\subsection {Proof of the First Equivalence}
\label {sub:first}
\quad

\begin {theorem} 
[$\probmin{} \equiv \probmin1$]
\label {thm:first}
Under hypotheses \hypotheses 5, there is a neighborhood of $x_0$ where both optimization problems $\probmin{}$ and $\probmin1$ of Table \ref {tab:equivalences} are well defined. Their values are asymptotically equal at $x_0$ in the sense of definition \ref {def:rational}.
\end {theorem}

\begin {proof}
By lemma \ref {lem:Lipschitz}, $x_0$ has a neighborhood $\Lipschitz$ where problem $\probmin{}$ is well defined for every $x \in \Lipschitz$, and the optimal value, $\MuF (x)$, is Lipschitz continuous at $x_0$ with constant $L$.

By corollary \ref {cor:Lipschitz} similarly, $x_0$ has a neighborhood $\LipschitzForOther1$ where problem $\probmin1$ is well defined for every $x \in \LipschitzForOther1$, and the optimal value, $\Mu1 (x)$, is Lipschitz continuous at $x_0$ with constant $L_1$.

Let $\ball {y_0} {r(\epsilon)/(1+\epsilon)} \times \colocated$ be the neighborhood of $(y_0, x_0)$ in lemma \ref {lem:colocated}, and let
\begin {displaymath}
N(\epsilon) = \Lipschitz \, \cap \, \LipschitzForOther1 \, \cap \, \colocated \, \cap \, \Ball {x_0} {\min \, \{ L^{-1}, \, L_1^{-1} \} { r(\epsilon) \over 1+\epsilon}} \, .
\end {displaymath} 
Note the ball in this formula is around $x_0$ rather than $y_0$.

Suppose $x \in N (\epsilon)$. Let $\MuF (x)$ be attained at $y$. By lemma \ref {lem:Lipschitz} and $x \in \ball {x_0} {L^{-1} r(\epsilon)$ $/(1+\epsilon)}$, therefore
\begin {displaymath}
\| y - y_0 \| = \MuF (x) \le L \| x - x_0 \| < r(\epsilon)/(1+\epsilon) \, ,
\end {displaymath}
which places $(y, x) \in \ball {y_0} {r(\epsilon)/(1+\epsilon)} \times \colocated$. Part 1 of lemma \ref {lem:colocated} now asserts there is a $\yone \in \ball {y_0} {r(\epsilon)}$ with
\begin {displaymath}
\| \yone - y \| \le \epsilon \, \| y - y_0 \| \quad \mbox
{and} \quad
\TanY (\yone, x) = F (y, x) = 0 \, .
\end {displaymath}
Thus
\begin {displaymath}
\Mu1 (x) \le \| \yone - y_0 \| 
\le \| \yone - y \| + \| y - y_0 \|
\le (1 + \epsilon) \, \| y - y_0 \|
= (1 + \epsilon) \, \MuF (x)
\end {displaymath}
which is the upper side of (\ref {eqn:rational}) in definition \ref {def:rational}. The inequality with $\MuF$ and $\Mu1$ exchanged is established by the same argument using $L_1$ instead of $L$, corollary \ref {cor:Lipschitz} instead of lemma \ref {lem:Lipschitz}, and lemma \ref {lem:colocated} part 2 instead of part 1. The two upper-side inequalities imply (\ref {eqn:rational}).
\end {proof}

\section {Equalities for the Dual Problems}
\label {sec:duality}

The duality theory for best linear approximation guarantees that the three pairs of dual problems in Table \ref {tab:equivalences} have equal values. Equalities like these are well known and can be established in many ways. These are derived from the following duality theorem that Luenberger \cite [p.\ 119, thm.\ 1] {Luenberger1969} proves directly from the Hahn-Banach theorem. 

\begin {theorem} 
[Best linear approximation]
\label {thm:best}
If $\Set$ is a subspace and $y_0$ is an element of a real, normed linear space, then
\begin {displaymath}
\Inf {$y \in \Set$} \| y - y_0 \| \; = \Max {$f \in \, \Set^\perp, \; \| f \| \le 1$} f (y_0) \, .
\end {displaymath}
\end {theorem}
\vspace {-1ex}
\begin {corollary}
[Best affine approximation]
\label {cor:best} 
If $\mathcal A$ is an affine subspace and $y_0$ is an element of a real, normed linear space, then
\begin {displaymath}
\Inf {$y \in \mathcal A$} \| y - y_0 \| \; = \Max {$f \in ({\mathcal
A}-a)^\perp, \; \| f \| \le 1$} f (y_0 - a)
\end {displaymath}
in which $a$ is any element of $\mathcal A$.
\end {corollary}
\begin {proof} 
Replace $y$, $y_0$, $\Set$ in theorem \ref {thm:best} by $y - a$, $y_0 - a$, ${\mathcal A} - a$. 
\end {proof}

\begin {corollary} 
\label {cor:matrix} 
Let $T : \Reals^m \rightarrow \Reals^p$ be a linear transformation. For every $y_0 \in \Reals^m$, each optimization problem below is well defined if and only if $h \in T (\Reals^m)$, in which case the optimal values are equal.
\begin {displaymath}
\Min {$y \in \Reals^m$ : $T y = h$} \| y - y_0 \| \; = \; \Max {$g \in
(\Reals^n)^*$ : $\| T^* g \| \le 1$} g (T y_0 - h)
\end {displaymath}
\end {corollary}

\vspace {-1ex}
\begin {proof} The minimization is well-posed whenever $h$ is in the image of $T$. The same can be proved for the maximization. If $h \in T (\Reals^m)$, then $h = Tu$ for some $u$, so the objective function,
\begin {displaymath}
g (T y_0 - h) = g T (y_0 - u) = (T^* g) (y_0 - u) \le \| T^* g \| \, \| y_0 - u \| \le \| y_0 - u \| \, , 
\end {displaymath}
is bounded above for every $g \in (\Reals^n)^*$. The maximum is attained because the feasible set is closed in a finite dimensional space.

Conversely, suppose the maximization is well posed. If $g \in T (\Reals^n)^\perp = \ker (T^*)$, then $g$ and all its multiples are feasible. Hence $g (h) = 0$, lest by scaling $g$ it would be possible to make $g(T y_0 - h) = g(h)$ arbitrarily large. Thus $h \in {}^\perp [T (\Reals^m)^\perp] = T (\Reals^m)$.

All that remains is to establish the equality using corollary \ref {cor:best}. Choose ${\mathcal A} = \{ y \in \Reals^m : T y = h \}$ and $a \in {\mathcal A}$. Now ${\mathcal A} - a = \ker (T)$, so 
\begin {displaymath}
({\mathcal A} - a)^\perp = [\ker (T)]^\perp = [{}^\perp (T^* (\Reals^n)^*) ]^\perp = T^* (\Reals^n)^* \, .
\end {displaymath}
This means $f \in ({\mathcal A} - a)^\perp$ if and only if $f = T^* g$ for some $g \in (\Reals^n)^*$. Thus the maximization in corollary \ref {cor:best} is over all such $g$ with $\| T^* g \| = \| f \| \le 1$. Finally, the objective function is
\begin {displaymath}
f(y_0 - a) = (T^* g) (y_0 - a) = g T (y_0 - a) = g(T y_0 - T a) = g(T y_0 - h).
\end {displaymath}
\end {proof} 

\begin {theorem} 
[$\probmin i \equiv \probmax i$, $i = 1, 2, 3$] 
\label {thm:duality}
Under hypotheses \hypotheses 5, there is a neighborhood of $x_0$ where problems $\probmin1$ and $\probmax1$ of Table \ref {tab:equivalences} are well defined and their values are equal, and similarly for the $\prob2$ and $\prob3$ pairs of dual problems.
\end {theorem}

\begin {proof} 
By lemma \ref {lem:ml}, $\d1 F (y_0, x)$ is onto for every $x \in \ml$. Therefore by corollary \ref {cor:matrix} the following problems are well defined and their values are equal for every $h \in \Reals^p$.
\begin {displaymath}
\Min {$y : \d1 F (y_0, x) y - h = 0$} \| y - y_0 \| \; = \; \Max {$f :
\| \d1 F (y_0, x)^* f \| \le 1$} f (\d1 F (y_0, x) y_0 - h) 
\end {displaymath}
Choosing $h = \d1 F (y_0, x) y_0 - F (y_0, x)$ gives the conclusion of the theorem for the $\prob1$ dual problems.

In particular $\d1 F (y_0, x_0)$ is onto, so also by corollary \ref {cor:matrix} the following problems are well defined and their optimal values are equal for every $h \in \Reals^p$.
\begin {displaymath}
\Min {$y : \d1 F (y_0, x_0) y - h = 0$} \| y - y_0 \| \; = \; \Max {$f :
\| \d1 F (y_0, x_0)^* f \| \le 1$} f (\d1 F (y_0, x_0) y_0 - h) 
\end {displaymath}
The choice $h = \d1 F (y_0, x_0) y_0 - F (y_0, x)$ gives the conclusion for the $\prob2$ dual problems; similarly $h = \d1 F (y_0, x_0) y_0 - \d2 F (y_0, x_0) (x - x_0)$ for the $\prob3$ problems.
\end {proof}

\section {Second Equivalence, $\probmax1 \equiv \probmax2$}
\label {sec:second}

The second equivalence to be proved, in the notation of Table \ref {tab:equivalences}) says that the feasible set $\{ f : \| \d1 F (y_0, x)^* f \| \le 1 \}$ can be replaced by one that is independent of $x$. The proof is self-contained and is the second of the two most complicated proofs in this paper.

\begin {theorem}
[$\probmax1 \equiv \probmax2$] 
\label {thm:second} 
Under hypotheses \hypotheses 5, there is a neighborhood of $x_0$ where both optimization problems $\probmax1$ and $\probmax2$ of Table \ref {tab:equivalences} are well defined. Their values are asymptotically equal at $x_0$ in the sense of definition \ref {def:rational}.
\end {theorem}

\begin {proof} 
The hypotheses suffice to invoke theorem \ref {thm:duality} which says $\probmax1$ and $\probmax2$ are well defined on some neighborhood $N^{(1)}$ of $x_0$. The feasible sets are given by
\begin {displaymath}
{\mathcal C} (x) = \{ \mbox {$f \in (\Reals^m)^*$ : $\| \d1 F (y_0, x)^* f \| \le 1$}\} 
\end {displaymath}
\begin {displaymath}
\Mu1 (x)_{\max} = \Max {$f \in {\mathcal C} (x)$} f(F(y_0, x))
\qquad
\Mu2 (x)_{\max} = \Max {$f \in {\mathcal C} (x_0)$} f(F(y_0, x))
\end {displaymath}
The proof has three steps that culminate in equations (\ref {eqn:step1}), (\ref {eqn:step2}) and (\ref {eqn:step3}), respectively. 

(Step 1.) If $f_1 \in \bd {{\mathcal C} (x_0)} = \{ \mbox {$f$ : $\| \d1 F (y_0, x_0)^* f \| = 1$}\}$, then
\begin {eqnarray}
\nonumber
\big| \, \| \d1 F (y_0, x)^* f_1 \| - 1 \big|& =& \big| \strut \| \d1 F (y_0, x)^* f_1 \| - \| \d1 F (y_0, x_0)^* f_1 \| \big|\\
\noalign {\smallskip}
\nonumber
& \le& \| \d1 F (y_0, x)^* f_1 - \d1 F (y_0, x_0)^* f_1 \|\\
\noalign {\smallskip}
\label {eqn:uniform}
& \le& \| \d1 F (y_0, x)^* - \d1 F (y_0, x_0)^* \| \, \| f_1 \|\\
\noalign {\smallskip}
\nonumber
& =& \| \d1 F (y_0, x) - \d1 F (y_0, x_0) \| \, \| f_1 \|\\
\noalign {\smallskip}
\nonumber
& \le& \| \d1 F (y_0, x) - \d1 F (y_0, x_0) \| \, \Max {$f \in \bd
{{\mathcal C} (x_0)}$} \| f \| \, .
\end {eqnarray}
The linear transformation $\d1 F (y_0, x_0)$ is onto, so its adjoint $\d1 F (y_0, x_0)^*$ is one-to-one. Hence $\| \d1 F (y_0, x)^* f \|$ defines a norm on the dual space whose closed unit ball is ${\mathcal C} (x_0)$. Thus, in the last bound of equation (\ref {eqn:uniform}), the maximum is finite because ${\mathcal C} (x_0)$ is compact. There also, the difference term converges to $0$ as $x \rightarrow x_0$ because $F$ is continuously differentiable. Altogether, $\| \d1 F (y_0, x)^* f_1 \|$ converges to $1$ uniformly on $\bd {{\mathcal C} (x_0)}$ as $x \rightarrow x_0$. This means, for every $\epsilon > 0$, there is a neighborhood $N^{(2)} (\epsilon)$ of $x_0$, such that
\begin {equation} \label {eqn:step1}
\mbox {$x \in N^{(2)} (\epsilon)$ and $f_1 \in \bd {{\mathcal C}
(x_0)}$} \; \Rightarrow \; 1 - \epsilon \le \| \d1 F (y_0, x)^* f_1 \|
\le 1 + \epsilon
\end {equation}

(Step 2.) Choose $x \in N^{(1)} \cap N^{(2)} (\epsilon)$, and then choose any nonzero $f \in {\mathcal C} (x)$, and finally let $f_1 = f / \| \d1 F (y_0, x_0)^* f \| \in \bd {{\mathcal C} (x_0)}$. Assume without loss of generality that $\epsilon < 1$. It is now possible to calculate
\begin {eqnarray*}
\| (1 - \epsilon) \d1 F (y_0, x_0)^* f \|
& =& (1 - \epsilon) \, \| \d1 F (y_0, x_0)^* f \|\\
& \le& \| \d1 F (y_0, x)^* f_1 \| \, \| \d1 F (y_0, x_0)^* f \|
\note {from (\ref {eqn:step1})}\\
& =& \| \d1 F (y_0, x)^* f \|\\
& \le& 1 
\note {because $f \in {\mathcal C} (x)$.} 
\end {eqnarray*}
This proves $(1 - \epsilon) f \in {\mathcal C} (x_0)$. Similarly, choose any nonzero $f_2 \in {\mathcal C} (x_0)$ and let $f_1 = f_0 / \| \d1 F (y_0, x_0)^* f_0 \| \in \bd {{\mathcal C} (x_0)}$. It now follows that
\begin {eqnarray*}
\| (1 + \epsilon)^{-1} \d1 F (y_0, x)^* f_0 \|
& =& (1 + \epsilon)^{-1} \, \| \d1 F (y_0, x)^* f_0 \|\\
& =& (1 + \epsilon)^{-1} \, \| \d1 F (y_0, x)^* f_1 \| \, \| \d1 F
(y_0, x_0)^* f_0 \|\\
& \le& \| \d1 F (y_0, x_0)^* f_0 \|
\note {from equation (\ref {eqn:step1})}\\
& \le& 1
\note {because $f_0 \in {\mathcal C} (x_0)$.}
\end {eqnarray*}
This proves $(1 + \epsilon)^{-1} f_0 \in {\mathcal C} (x)$. These two calculations establish the next implication.
\begin {equation} \label {eqn:step2}
x \in N^{(1)} \cap N^{(2)} (\epsilon) \; \Rightarrow \; (1 - \epsilon)
\, {\mathcal C} (x) \; \subseteq \; {\mathcal C} (s_0) \; \subseteq \;
(1 + \epsilon) \, {\mathcal C} (x)
\end {equation}

(Step 3.) Choose $x \in N^{(1)} \cap N^{(2)} (\epsilon)$, and then choose $f_1 \in {\mathcal C} (x)$ that attains $\Mu1 (x)_{\max}$. Equation (\ref {eqn:step2}) asserts $(1 - \epsilon) f_1 \in {\mathcal C} (x_0)$, so
\begin {displaymath}
\Mu2 (x)_{\max} \; = \Max {$f \in {\mathcal C} (x_0)$} f (F(y_0, x)) \; \ge \; (1 - \epsilon) f_1 (F(y_0, x)) \; = \; (1 - \epsilon) \, \Mu1 (x)_{\max} \, .
\end {displaymath}
Similarly, choose $f_2 \in {\mathcal C} (x_0)$ that attains $\Mu2 (x)_{\max}$. Now equation (\ref {eqn:step2}) asserts $(1 + \epsilon)^{-1} f_2 \in {\mathcal C} (x)$, so
\begin {displaymath}
\Mu1 (x)_{\max} \; = \Max {$f \in {\mathcal C} (x)$} f (F(y_0, x)) \;
\ge \; (1 + \epsilon)^{-1} f_2 (F(y_0, x)) \; = \; (1 + \epsilon)^{-1}
\Mu2 (x)_{\max} \, .
\end {displaymath}
Together these two inequalities provide the final implication,
\begin {equation} \label {eqn:step3}
x \in N^{(1)} \cap N^{(2)} (\epsilon) \; \Rightarrow \; (1 - \epsilon)
\, \Mu1 (x)_{\max} \le \Mu2 (x)_{\max} \le (1 + \epsilon) \Mu1
(x)_{\max} \, ,
\end {equation}
which is (\ref {eqn:rational}) in definition \ref {def:rational}.
\end {proof}

\section {Third Equivalence, $\probmax2 \equiv \probmax3$}
\label {sec:third}

The proof of the last equivalence involves a class of norms that has been used already in the proof of theorem \ref {thm:second}. If a linear mapping $T : \Reals^m \rightarrow \Reals^p$ is onto, then its adjoint $T^*$ is one-to-one, so $\| T^* f \|$ defines a norm on the dual space, $(\Reals^p)^*$. The dual of this norm, viewed as a norm on $\Reals^p$, is given by the following construction. All the maximization problems in Table \ref {tab:equivalences} are norms of this kind. 

\begin {lemma} 
\label {lem:norm}
If a linear transformation $T : \Reals^m \rightarrow \Reals^p$ is onto, then
\begin {displaymath}
\| v \|_T \, := \Max {$f : \| T^* f \| \le 1$} \, f(v) \, ,
\end {displaymath}
is a norm on $\Reals^p$. (The proof is clear.)
\end {lemma}

\begin {theorem}
[$\probmax2 \equiv \probmax3$] 
\label {thm:third} 
Under hypotheses \hypotheses 5, there is a neighborhood of $x_0$ where both of optimization problems $\probmax2$ and $\probmax3$ of Table \ref {tab:equivalences} are well defined. Their values are differentially equivalent at $x_0$ in the sense of definition \ref {def:differential}.

Under hypotheses \hypotheses 6, the values of the problems $\probmax2$ and $\probmax3$ are asymptotically equal at $x_0$ in the sense of definition \ref {def:rational}.
\end {theorem}

\begin {proof}
(Part 1.) Let $\| \cdot \|_T$ be the norm given in lemma \ref {lem:norm} for the linear transformation $T = \d1 F (y_0, x_0)$. Let $\TanX (y, x) = \d2 F (y, x_0) (x - x_0) + F (y, x_0)$ be the linear function parameterized by $y$ whose graph is tangent to the graph of $F (y, x)$ at $x = x_0$. (Note this is not the $\TanY$ of Table \ref {tab:simplifications}.) In this notation, $\Mu2 (x)_{\max} = \| F (y_0, x) \|_T$ and $\Mu3 (x)_{\max} = \| \TanX (y_0, x) \|_T$. Thus by the triangle inequality,
\begin {displaymath}
\big| \, \Mu2 (x)_{\max} - \Mu3 (x)_{\max} \big|
= \big| \, \strut \| F (y_0, x) \|_T - \| \TanX (y_0, x) \|_T \big|
\le \| F (y_0, x) - \TanX (y_0, x) \|_T \, .
\end {displaymath}
The difference between $F (y_0, x)$ and $\TanX (y_0, x)$ is $o ( \| x - x_0 \|)$ uniformly in $x$ by the definition of Fr\'echet differentiability. The same estimate applies in the $\| \cdot \|_T$ norm because all norms are equivalent in finite dimensional spaces. Therefore
\begin {displaymath}
\Lim {$x \rightarrow x_0$} {| \strut \Mu2 (x)_{\max} - \Mu3 (x)_{\max} | \over \| x - x_0 \|} \; \le \; \Lim {$x \rightarrow x_0$} {\| F(y_0, x) - \TanX (y_0, x) \|_T \over \| x - x_0 \|} \; = \; 0 \, ,
\end {displaymath}
which is (\ref {eqn:FrechetEquivalent}) in definition \ref {def:differential}.

(Part 2.) The Fr\'echet differentiability of $F$ with $F(y_0, x_)0) = 0$ imply
\begin {displaymath}
F (y_0, x) = \d2 F (y_0, x_0) (\Delta x) + R (\Delta x)
\end {displaymath}
where $\Delta x = x - x_0$ and the remainder $R(\Delta x)$ is $o(\| \Delta x \|)$. Again by the triangle inequality,
\begin {displaymath}
\| F (y_0, x) \|_T \, = \, \| \d2 F (y_0, x_0) (\Delta x) + R(\Delta x) \|_T \hspace {0.5em} \raisebox {-1ex} {$\displaystyle \mathop {<}^{\displaystyle >}$} \hspace {0.5em} \| \d2 F (y_0, x_0) (\Delta x) \|_T \; \raisebox {-1ex} {$\displaystyle \mathop {+}^{\displaystyle -}$} \; \| R(\Delta x) \|_T \, .
\end {displaymath}
where as noted $\Mu2 (x)_{\max} = \| F (y_0, x) \|_T$ and $\Mu3 (x)_{\max} = \| \d2 F (y_0, x_0) (\Delta x) \|_T$. The latter is a norm for $\Delta x$ under the present hypothesis that $\d2 F(y_0, x_0)$ is one-to-one. Thus, if $x \ne x_0$, then the inequalities can be divided by $\Mu3 (x)$ to give,
\begin {displaymath}
\left| {\Mu2 (x)_{\max} \over \Mu3 (x)_{\max}} - 1 \, \right| \le {\| R(\Delta x) \|_T \over \| \d2 F (y_0, x_0) (\Delta x) \|_T} \, .
\end {displaymath}
Again by the equivalence of all norms for a finite dimensional space, the upper bound vanishes in the limit $x \rightarrow x_0$ because $R(\Delta x)$ is $o(\| \Delta x \|)$. The vanishing limit implies (\ref {eqn:rational}) in definition \ref {def:rational}.
\end {proof}

\section {Summary in Matrix Notation and for 2-Norms}

Suppose bases have been chosen for $\Reals^m$, $\Reals^n$, $\Reals^p$ and a norm has been chosen to measure perturbations in $\Reals^m$. These choices express the optimization problems in matrix notation: 
\smallskip
\begin {enumerate}
\item $\J(x)$ is the $p \times n$ Jacobian matrix for $\d2 F (y_0, x)$. The entries are the partial derivatives of $F(y,x)$ with respect to $x$ evaluated at $(y_0, x)$.
\item $\K (x)$ is the $p \times m$ Jacobian matrix for $\d1 F (y_0, x)$. Entries are partial derivatives of $F(y,x)$ with respect to $y$ evaluated at $(y_0, x)$.
\item The residual vector of the equations is $\r (x) = F(y_0, x) \in \Reals^p$.
\item $\| \cdot \|$ is the chosen norm for $\Reals^m$, and $\| \cdot \|^*$ is the dual norm.
\end {enumerate}
\smallskip
The matrix versions of the problems are in \textbf {Table \ref {tab:matrix}}. If $\K (x_0)$ has full row rank, then this paper has shown: 
\smallskip
\begin {enumerate}
\item The minimizations and maximizations of Table \ref {tab:matrix} are duals (theorem \ref {thm:duality}).
\item The optimal values $\MuF {(x)}$, $\Mu 1 (x)$, $\Mu 2 (x)$ are asymptotically equal at $x_0$ (theorems \ref {thm:first} and \ref {thm:second}).
\item $\Mu 3 (x)$ is differentially equivalent to the other values (theorem \ref {thm:third}).
\end {enumerate} 
\smallskip
That is, the values $\Mu i (x)$ approximate $\MuF (x)$ increasingly well as $x$ nears $x_0$. For $2$-norms, the approximations can be found very simply using the matrix QR factorization.

\begin {table} [h]
\caption {Optimization problems of Table \ref {tab:equivalences} in matrix notation. In these formulas, $\Delta x = x - x_0$.}
\label {tab:matrix}
\small
\renewcommand {\substrut} {{\normalsize \vrule depth0pt height2ex
width0pt}}
\renewcommand {\Max} [1] {\max_{\makebox[4em][l]{\substrut \scriptsize #1}}}
\renewcommand {\Min} [1] {\min_{\makebox[6em][l]{\substrut \scriptsize #1}}}
\newcommand {\minobject} {\| \Delta y \|}
\newcommand {\tablestrut} {\vrule depth4.25ex height3.0ex width0pt}
\arraycolsep = 0.75em
\begin {displaymath}
\begin {array} {| c | c | l | l |}
\mbox {name}& \mbox {value}& \multicolumn{1}{c|}{\mbox {minimization form}}& \multicolumn{1}{c |}{\mbox {dual maximization}}\\
\hline \hline
\tablestrut \prob{}&
\MuF (x)&
\displaystyle 
\Min {$F (y, x) = 0$} \| y - y_0 \|&
\\ \hline
\tablestrut \prob1&
\Mu1 (x)&
\displaystyle 
\Min {$\K (x) \, \Delta y = {} - \r (x)$} \minobject&
\displaystyle 
\Max {$\| \, \K (x)^T u \, \|^* \le 1$} u^T \r (x)
\\ \hline
\tablestrut \prob2&
\Mu2 (x)&
\displaystyle 
\Min {$\K (x_0) \, \Delta y = {} - \r (x)$} \minobject&
\displaystyle 
\Max {$\| \, \K (x_0)^T u \, \|^* \le 1$} u^T \r (x)
\\ \hline
\tablestrut \prob3&
\Mu3 (x)&
\displaystyle
\Min {$\K(x_0) \, \Delta y = {} - \J (x_0) \, \Delta x$} 
\minobject&
\displaystyle \Max {$\| \, \K (x_0)^T u \, \|^* \le 1$} u^T \J (x_0) \, \Delta x
\\ \hline
\end {array}
\end {displaymath}
\end {table}

\begin {lemma}
Let $A \in \Reals^{m \times p}$ and $s, u \in \Reals^p$. If $A$ has full column rank, then for the $A = QR$ factorization, 
\begin {displaymath}
\Max {$\| \kern 0.5pt A u \kern 1pt \|_2 \le 1$} u^T s = \| \kern 1pt R^{-T} s \kern 1pt \|_2 \, .
\end {displaymath}
\end {lemma}

\begin {proof}
Because $u^T s = u^T R^T R^{-T} s$ and $\| R u \|_2 = \| A u \|_2 = 1$ therefore $u^T s \le \| R^{-T} s \|_2$ with equality when $u = R^{-1} R^{-T} s / \| R^{-T} s \|_2$.
\end {proof}

\appendix

\addcontentsline {toc} {section} {Appendix. Nomenclature and Notation}
\section {Nomenclature and Notation} This appendix lists some standard notation that is used without comment throughout the paper.

\newcommand {\face} [1] {\textbf {#1}}

\smallskip
\begin {enumerate}
\item For $f : {\mathcal D} \subseteq \Reals^m \times \Reals^n \rightarrow \Reals^p$, the \face {Fr\'echet derivative} of $f$ evaluated at $(y, x)$ is $D f (y, x) \in \hom (\Reals^{m+n}, \Reals^p)$. The \face {partial Fr\'echet derivative} of $f$ with respect to the first space $\Reals^m$ and evaluated at $(y, x)$ is $\d1 f (y, x) \in \hom (\Reals^m, \Reals^p)$, and similarly for the second space $\Reals^n$ and $\d2 f$. 

\smallskip
\item The \face {dual space} of a normed linear space $\Reals^m$ is the space of functionals $(\Reals^m)^* = \hom (\Reals^m, \Reals)$ with the induced norm. The annihilator of a set $\Set \subseteq \Reals^m$ is the subspace $\Set^\perp \subseteq (\Reals^m)^*$. The subspace annihilated by a set $\Set \subseteq (\Reals^m)^*$ is ${}^\perp\Set \subseteq \Reals^m$. The \face {transpose} of $T \in \hom (\Reals^m, \Reals^p)$ is $T^* \in \hom ((\Reals^p)^*, (\Reals^m)^*)$.

\smallskip
\item The \face {interior}, \face {boundary}, and \face {closure} of a set $\Set$ are indicated by $\Int{\Set}$, $\bd{\Set}$, and $\cl{\Set}$.

\smallskip
\item The \face {open ball} with center $c$ and radius $r$ is $\ball c r$.

\smallskip
\item Six lemmas assert the existence of \face {neighborhoods} that are indicated by placing the lemma number in a superscript, the point around which the neighborhood lies in a subscript, and any parameterization of the neighborhood in parentheses: 
\begin {displaymath}
\Lipschitz 
\quad 
\LipschitzForOther i
\quad 
\MVTs \epsilon 
\quad 
\MVTd 
\epsilon 
\quad
\ml 
\quad
\colocated \, .
\end {displaymath}
\end {enumerate}


\addcontentsline {toc} {section} {References}

\end{document}